\documentclass[11pt]{amsart} \textwidth=14.5cm \oddsidemargin=1cm
\usepackage{amsmath,wasysym}
\usepackage{upgreek}
\usepackage{amsxtra}
\usepackage{amscd}
\usepackage{amsthm}
\usepackage{amsfonts}
\usepackage{amssymb}
\usepackage{eucal}
\usepackage{graphics, color}
\usepackage{hyperref,mathrsfs}
\usepackage[usenames,dvipsnames]{xcolor}
\usepackage{tikz}
\usepackage{stmaryrd}
\usepackage[all]{xy}
\usepackage{hyperref}
\usepackage{color}
\usepackage{bbm} 
\allowdisplaybreaks

\hypersetup{colorlinks,linkcolor=blue, urlcolor=blue,
citecolor=blue}

\newtheorem{theorem}{Theorem}[section]
\newtheorem{lemma}[theorem]{Lemma}
\newtheorem{proposition}[theorem]{Proposition}
\newtheorem{corollary}[theorem]{Corollary}

\theoremstyle{definition}
\newtheorem{definition}[theorem]{Definition}
\newtheorem{example}[theorem]{Example}

\theoremstyle{remark}
\newtheorem{remark}[theorem]{Remark}

\numberwithin{equation}{section}

\usepackage{shuffle}

\title[A quantum shuffle approach to quantum determinants]
{A quantum shuffle approach to quantum determinants}

\author[Run-Qiang Jian]{Run-Qiang Jian}
\address{School of Computer Science, Dongguan University
of Technology, 1, Daxue Road, Songshan Lake, 523808, Dongguan, P.
R. China}
\email{jianrq@dgut.edu.cn}

\dedicatory{To Professor Marc Rosso on the occasion of his 60th birthday}

\keywords{Quantum shuffle product, quantum exterior algebra, FRT algebra, convolution product, quantum determinant}
\subjclass[2020]{17B37, 20G42}

\begin{document}

\begin{abstract}
Let $\bigwedge_\sigma V=\bigoplus_{k\geq 0}\bigwedge_\sigma^kV$ be the quantum exterior algebra associated to a finite-dimensional braided vector space $(V,\sigma)$. For an associative algebra $\mathfrak{A}$, we consider the convolution product on the graded space $\bigoplus_{k\geq 0}\mathrm{Hom}_{\mathbb{C}}\big(\bigwedge_\sigma^kV,\bigwedge_\sigma^kV\otimes \mathfrak{A}\big)$.  Using this product, we define a notion of quantum minor determinant of a map from $V$ to $V\otimes \mathfrak{A}$, which coincides with the classical one in the case that $\mathfrak{A}$ is the FRT algebra corresponding to $U_q(\mathfrak{sl}_N)$. We establish a quantum Laplace expansion formula and multiplicative formula for these determinants.
\end{abstract}

\maketitle
\setcounter{tocdepth}{1}

\section{Introduction}
Influenced by the work of Faddeev and his St. Petersburg school, Drinfeld \cite{D} and Jimbo \cite{Jim1} introduced quantum groups independently in the mid-1980s. After their birth, quantum groups have been developed greatly and become one of the most important subjects in mathematics and mathematical physics over the past 35 years. Quantum groups include two important classes of examples. One consists of quantized enveloping algebras while the other consists of quantum matrix algebras. Quantum matrix algebras and their quotients were constructed by Faddeev, Reshetikhin, and Takhtajan \cite{FRT}, and Manin \cite{Ma}. They are dual to quantized enveloping algebras and can be viewed as some sort of quantization of the coordinate rings of linear algebraic groups. Quantum determinants, or more generally quantum minor determinants, play an essential role in the investigation of quantum matrix algebras. They were defined firstly for Jimbo's R-matrix of $U_q\mathfrak{gl}_N$ \cite{RTF, Ma}, and soon exteded to even Hecke symmetries \cite{Gu}. Recently, they are constructed for general FRT algebras\cite{FG}, and for quantum matrix algebras related to couples of compatible braidings \cite{GS}. Many significant results and applications of quantum determinants have been found during the last three decades, such as constructions of quantum linear groups \cite{RTF,PW}, the Cayley-Hamilton theorem \cite{Zh,IOPS}, zonal spherical functions \cite{NYM}, the quantum Capelli identity \cite{NUW}, and quantum Phaffians \cite{JZ}.

In the case of type A, one usually uses quantum exterior algebras to derive results about quantum determinants. Many results mentioned above use such a technique. In the present paper, we attempt to extend the machinery in type A to general cases. Our starting point is quantum exterior algebras related to general braidings. General quantum exterior algebras were constructed by Woronowicz in his work on covariant differential calculus on quantum groups \cite{Wo}. Rosso \cite{Ro} generalized them in the framework of quantum shuffle algebras for arbitrary braidings. These algebras are also called Nichols algebras in the category of Yetter-Drinfeld modules. Besides their own interest (see, e.g., \cite{FCG, IO}), quantum exterior algebras have many important applications (see, e.g.,\cite{Ro, He, AS1,AS2}). We will use the bialgebra structure of quantum exterior algebras to define quantum determinants and study their properties.

Let us give a brief description of our construction. In linear algebra, one way to define minor determinants is to use exterior products of endomorphisms. Given a vector space $V$, the exterior product of two maps $f\in \mathrm{End}(\bigwedge^m V)$ and $g\in \mathrm{End}(\bigwedge^n V)$ is given by \begin{align*}\lefteqn{(f \wedge g)(u_1\wedge\cdots \wedge u_{m+n})}\\
&=\sum_{w}\mathrm{sign}(w)f(u_{w(1)}\wedge\cdots\wedge u_{w(m)})\wedge g(u_{w(m+1)}\wedge\cdots\wedge u_{w(m+n)}),\end{align*} where the sum runs through all $(m,n)$-shuffles, and $u_1,\ldots,u_{m+n}\in V$. The above formula involves two structures: the exterior product and the shuffle action. The later one can be viewed as a coproduct on the exterior algebra built on $V$. In order to extend minor determinants to quantum case, we replace the usual exterior algebras by certain quantum exterior algebras and consider some sort of quantum exterior product of maps. By combining all these structures, we define quantum minor determinants of an endomorphism which take values in an arbitrary algebra. Some properties of usual determinants and classical quantum determinants are extended to these new quantum minor determinants.

Finally, we indicate some advantages of our approach. First of all, our interpretation of quantum determinants matches the original one of linear endomorphisms. Secondly, this approach exhibits that comodule structures, product, and coproduct of quantum exterior algebras are essential ingredients of quantum determinants. The last one is that our framework is more flexible for applications. We work on general bialgebras, but not just FRT algebras.

This paper is organized as follows. In Section 2, some notation and terminologies are fixed. In Section 3, we recall the construction of quantum exterior algebra associated to a braided space. In Section 4, we show that quantum exterior algebras possess comodule structures over Faddeev-Reshetikhin-Takhtajan bialgebras. As an application, we show that the positive part of a quantum group admits such a comodule structure. In Section 5 and 6, we consider convolution product on the space of graded homomorphisms from a quantum exterior algebra to the tensor product of this algebra and an algebra. Quantum minor determinants are introduced by using this product. A quantum Laplace expansion formula and multiplicative formula are established for these determinants.

\section{Some notation}

We start by fixing some notation which will be used frequently in the sequel.

In this paper, the ground field is always $\mathbb{C}$, the field of complex numbers. All objects
we discuss are defined over $\mathbb{C}$. We always use small letters $i,j,k,l,m,n,p$ to denote nonnegative integers.

For $n\geq 1$, we denote by $\mathfrak{S}_{n}$ the symmetric group of the
set $\{1,2,\ldots,n\}$, and by $s_{i}$, $1\leq i\leq n-1$, the
transposition permuting $i$ and $i+1$. The length of a permutation $w\in \mathfrak{S}_{n}$ is denoted by $l(w)$, i.e., $l(w)=\sharp \{(i,j)|1\leq i<j\leq n, w(i)>w(j)\}$.  For any positive integers $i_1,i_2,\ldots,i_k$, an $(i_1,i_2,\ldots,i_k)$-shuffle is a permutation $w\in\mathfrak{S}_{i_1+\cdots+i_k}$ satisfying \begin{align*}
w(1)<w(2)&<\cdots<w(i_1),\\[3pt]
w(i_1+1)<w(i_1+2)&<\cdots<w(i_1+i_2),\\[3pt]
&\ \ \ \ldots, \\[3pt]
w(i_1+\cdots+i_{k-1}+1)<w(i_1+\cdots+i_{k-1}+2)&<\cdots<w(i_1+\cdots+i_{k}).
\end{align*}We denote by $\mathfrak{S}_{i_1,i_2,\ldots,i_k}$ the set of all $(i_1,i_2,\ldots,i_k)$-shuffles. As usual, the notation $\mathfrak{S}_{i_1}\times \mathfrak{S}_{i_2}\times \cdots\times \mathfrak{S}_{i_k}$ stands for the Young subgroup of $\mathfrak{S}_{i_1+\cdots+i_k}$ corresponding to the composition $(i_1,i_2,\ldots,i_k)$. We also denote by $1_{\mathfrak{S}_m}\times \mathfrak{S}_n$ the subgroup of $\mathfrak{S}_{m+n}$ whose elements fix $1,2,\ldots,m$. The definition of $\mathfrak{S}_m\times 1_{\mathfrak{S}_n}$ is similar. Some results below depend on special decompositions of permutations. Let $S_1,S_2,S_3$ be subsets of $\mathfrak{S}_n$. The ordered pair $(S_1, S_2)$ is called a \emph{reduced decomposition} of  $S_3$ if the multiplication map \[\begin{array}{cccc}
 & S_1\times S_2&\rightarrow&S_3 ,\\[3pt]
&(w_1,w_2)&\mapsto&w_1w_2,
\end{array}\]is bijective and moreover $l(w_1w_2)=l(w_1)+l(w_2)$ for each $w_1\in S_1$ and $w_2\in S_2$.

Let $\nu\in \mathbb{C}$ be a nonzero number which is not a root of unity. For any nonnegative integer $n$, we denote $(n)_\nu=\frac{1-\nu^n}{1-\nu}$, $(n)_\nu!=(1)_\nu(2)_\nu\cdots(n)_\nu$ for $n\geq 1$, and $(0)_\nu!=1$. We have the following formulas \begin{equation}\label{Quantum factorial}\sum_{w\in \mathfrak{S}_n}\nu^{l(w)}=(n)_\nu!\end{equation} and \begin{equation}\label{Quantum binomial}\sum_{w\in \mathfrak{S}_{m,n}}\nu^{l(w)}=\frac{(m+n)_\nu!}{(m)_\nu!(n)_\nu!}.\end{equation}The first one can be proved by an easy induction. For a proof of the second one, one can consult Theorem 6.1 in \cite{KC}.

In this paper, the composition of two linear maps $F$ and $G$ will be written as $FG$ if it is definable. For any two vector spaces $V$ and $W$, the flip map $\tau_{V,W}: V\otimes W\rightarrow W\otimes V$ is defined by $\tau_{V,W}(v\otimes w)=w\otimes v$. We often abbreviate $\tau_{V,W}$ to $\tau$ if $V$ and $W$ are clear in the context.

Let $\mathfrak{A}$ be an associative algebra with multiplication map $\mathfrak{m}$. Let $V$ be a vector space, and $T(V)=\mathbb{C}\oplus V\oplus V^{\otimes 2}\oplus \cdots$ the tensor algebra over $V$ whose product is given by concatenation. We endow $T(V)\otimes \mathfrak{A}$ with the following product $\diamond$: for any $x\in V^{\otimes m}$, $y\in V^{\otimes n}$, and $a,b\in \mathfrak{A}$, $$(x\otimes a)\diamond(y\otimes b)=(x\otimes y)\otimes (ab).$$ For a linear map $F:V\rightarrow V\otimes  \mathfrak{A}$ and any $k\geq 1$, we define the \emph{$k$-fold cross product} $F^{\times k}: V^{\otimes k}\rightarrow V^{\otimes k}\otimes \mathfrak{A}$ of $F$ recursively by $F^{\times 1}=F$, and \begin{equation}\label{k-fold of cross product}F^{\times (k+1)}=(\mathrm{id}_{V^{\otimes (k+1)}}\otimes \mathfrak{m})(\mathrm{id}_V\otimes \tau\otimes \mathrm{id}_{\mathfrak{A}})(F\otimes F^{\times k}),\end{equation} for $k\geq 1$. Obviously, we have \begin{equation}\label{Multiplicative property}F^{\times (m+n)}(x\otimes y)=\big(F^{\times m}(x)\big)\diamond\big(F^{\times n}(y)\big).\end{equation}Additionally, if $\mathfrak{A}$ is a bialgebra and $F$ is a right $\mathfrak{A}$-comodule structure map on $V$, then $F^{\times k}$ is just the usual tensor comodule structure map on $V^{\otimes k}$.

If $\mathfrak{A}$ is a bialgebra and $\varrho:V\rightarrow V\otimes  \mathfrak{A}$ is a comodule structure map on $V$, we adopt Sweedler's notation for comodules: for any $v\in V$, $$\varrho(v)=\sum_{(v)} v_{(0)}\otimes v_{(1)}=\sum v_{(0)}\otimes v_{(1)}.$$

\section{Quantum exterior algebras}
In this section, we recall the construction of quantum exterior algebras and some of their basic properties. For more details, we refer the reader to \cite{Ro,FCG,IO}.

We always assume that $V$ is an $N$-dimensional vector space with a fixed basis $\{v_1,\ldots,v_N\}$ throughout this paper. Denote by $I_k$ the identity map of $V^{\otimes k}$ for any nonnegative integer $k$. We sometimes write $I$ in place of $I_1$ for simplicity. A braiding $\sigma$ on $V$ is an invertible linear
map in $\mathrm{End}(V\otimes V)$ satisfying the braid relation: $$(\sigma\otimes
I)(I\otimes \sigma)(\sigma\otimes
I)=(I\otimes \sigma)(\sigma\otimes
I)(I\otimes \sigma).$$ We call the pair $(V,\sigma)$ a braided space. A braiding $\sigma$ is said to be of Hecke type if it satisfies the following Hecke relation $$(\sigma-qI_2)(\sigma+q^{-1}I_2)=0,$$or equivalently $$\sigma^2=(q-q^{-1})\sigma+I_2.$$ The most well-known braiding of Hecke type is undoubtedly Jimbo's braiding of type $A_{N-1}$ (\cite{Jim2}). It is defined by  \begin{equation}\label{Jimbo braiding}\sigma(v_i\otimes v_j)=\begin{cases}v_j\otimes v_i,&\text{if }i<j\\[3pt]qv_i\otimes v_i,&\text{if }i=j,\\[3pt]v_j\otimes v_i+(q-q^{-1})v_i\otimes v_j,&\text{if }i>j.\end{cases}\end{equation}For more constructions of braidings of Hecke type, one can consult \cite{Gu}.

For any positive integers $n\geq 2$ and $1\leq i\leq
n-1$, we denote by $\sigma_i$ the operator $I_{
i-1}\otimes \sigma\otimes I_{ n-i-1}\in
\mathrm{End}(V^{\otimes n})$. We must mention here that we only adopt this convention for braidings in our paper. For any other operators, subscripts have no such a meaning. For a reduced expression $w=s_{i_1}\cdots s_{i_l}$ of $w\in \mathfrak{S}_n$, we
define an operator $\sigma_w\in
\mathrm{End}(V^{\otimes n})$ by $$\sigma_w=\sigma_{i_1}\cdots \sigma_{i_l}.$$ By a well-known result of Matsumoto (see, e.g., Theorem 1.2.2 in \cite{GP}), $\sigma_w$ is independent of the choice of reduced expressions of $w$. In particular, $\sigma_i=\sigma_{s_i}$.

Let $q\in \mathbb{C}$ be a nonzero number such that $(n)_{q^{-2}}\neq 0$ for any $n\geq 1$ and $q^2\neq -1$. Given $k$ positive integers $i_1,\ldots,i_k$, we define two operators $$\shuffle_{i_1,\ldots,i_k}^\pm=\frac{(i_1)_{q^{-2}}!\cdots (i_k)_{q^{-2}}!}{(i_1+\cdots+i_k)_{q^{-2}}!}\sum_{w\in \mathfrak{S}_{i_1,\ldots,i_k}}(-q)^{-l(w)}\sigma_{w^{\pm 1}}.$$ The quantum shuffle product $\shuffle_\sigma$ on the tensor space $T(V)$ is defined as follows: for any $x\in V^{\otimes m}$ and $y\in V^{\otimes n}$, $$x\shuffle_\sigma y=\shuffle_{m,n}^+(x\otimes y).$$Then $(T(V),\shuffle_\sigma)$ is a unital associative algebra (see Proposition 9 in \cite{Ro}). Indeed, the associativity follows immediately from the fact that the pairs $(\mathfrak{S}_{i+j,k}, \mathfrak{S}_{i,j}\times 1_{\mathfrak{S}_{k}})$ and $(\mathfrak{S}_{i,j+k}, 1_{\mathfrak{S}_{i}}\times\mathfrak{S}_{j,k})$ are both reduced decompositions of $\mathfrak{S}_{i,j,k}$. This algebra is called the quantum shuffle algebra built on the braided space $(V,\sigma)$.

The graded subalgebra $\bigwedge_\sigma V$ of the quantum shuffle algebra $(T(V),\shuffle_\sigma)$ generated by $V$ is called the \emph{quantum exterior algebra} associated to the braided space $(V,\sigma)$. It is the quantum symmetric algebra associated to $-q^{-1}\sigma$ in the sense of Rosso \cite{Ro}. Let us provide a more explicit form of $\bigwedge_\sigma V$. The quantum anti-symmetrizers $A^{(k)}$ associated to $(V,\sigma)$ are defined as follows: $A^{(1)}=I\in \mathrm{End}(V)$ and $$A^{(k)}=\frac{1}{(k)_{q^{-2}}!}\sum_{w\in \mathfrak{S}_{k}}(-q)^{-l(w)}\sigma_w\in \mathrm{End}(V^{\otimes k})$$ for $k> 1$. For $m,n\geq 1$, we notice that the pair $(\mathfrak{S}_{m,n}, \mathfrak{S}_{m}\times \mathfrak{S}_{n})$ is a reduced decomposition of $\mathfrak{S}_{m+n}$. So we have $$\shuffle_{m,n}^+(A^{(m)}\otimes A^{(n)})=A^{(m+n)}.$$Hence for any $x\in V^{\otimes m}$ and $y\in V^{\otimes n}$, $$A^{(m)}(x)\shuffle_\sigma A^{(n)}(y)=A^{(m+n)}(x\otimes y).$$ As a consequence, $$\bigwedge\nolimits_\sigma V=\bigoplus_{k\geq 0} \bigwedge\nolimits_\sigma^k V,$$ where $\bigwedge_\sigma^k V=\mathrm{Im}A^{(k)}$ for $k\geq 1$ and $\bigwedge_\sigma^0 V=\mathbb{C}$.

Now let us look at braidings of Hecke type. In such a case the quantum anti-symmetrizers have many remarkable properties.

\begin{proposition}[Gurevich]If $\sigma$ is a braiding of Hecke type, then for any positive integers $k\geq 2$ and $1\leq i<k$, \begin{equation}\label{Gurevich1}A^{(k)}\sigma_i=\sigma_i A^{(k)}=-q^{-1}A^{(k)}.\end{equation} \end{proposition}

For a proof, one can see Proposition 2.1 in \cite{Gu}. As a consequence, we have

\begin{proposition}[Gurevich]If $\sigma$ is a braiding of Hecke type, then for any $k\geq 1$, \begin{equation}\label{Gurevich2}A^{(k)}A^{(k)}=A^{(k)}.\end{equation} \end{proposition}
\begin{proof}By (\ref{Gurevich1}), we have immediately that $$A^{(k)}\sigma_w=(-q)^{-l(w)}A^{(k)},$$for any $w\in \mathfrak{S}_k$. Hence\begin{align*}
A^{(k)}A^{(k)}&=\frac{1}{(k)_{q^{-2}}!}\sum_{w\in \mathfrak{S}_{k}}(-q)^{-l(w)}A^{(k)}\sigma_w\\[3pt]
&=\frac{1}{(k)_{q^{-2}}!}\sum_{w\in \mathfrak{S}_{k}}(-q)^{-2l(w)}A^{(k)}\\[3pt]
&=A^{(k)},
\end{align*}where the last equality follows from (\ref{Quantum factorial}).\end{proof}

\begin{lemma}\label{Lemma for kernel}If $\sigma$ is a braiding of Hecke type, then $A^{(2)}(x)=0$ for any  $x\in \mathrm{Ker}(\sigma-qI_2)$.\end{lemma}
\begin{proof}Since $x\in \mathrm{Ker}(\sigma-qI_2)$, we have $x=q^{-1}\sigma(x)$, and hence $$A^{(2)}(x)=q^{-1}A^{(2)}\sigma(x)=-q^{-2}A^{(2)}(x).$$Thus $(1+q^{-2})A^{(2)}(x)=0$ which implies $A^{(2)}(x)=0$.\end{proof}

\begin{proposition}\label{Quotient version of QEA}If $\sigma$ is a braiding of Hecke type, then the quantum exterior algebra $\bigwedge_\sigma V$ is isomorphic to the quotient algebra of the tensor algebra $T(V)$ by the ideal generated by $\mathrm{Ker}(\sigma-qI_2)$.\end{proposition}
\begin{proof}Denote by $\mathcal{I}$ the ideal of $T(V)$ generated by $\mathrm{Ker}(\sigma-qI_2)$ and by $\mathcal{I}_k$ the intersection of $\mathcal{I}$ and $V^{\otimes k}$. Due to a result of Gurevich (Proposition 2.13 in \cite{Gu}), $\mathrm{Im} A^{(k)}$ is isomorphic to $V^{\otimes k}/\mathcal{I}_k$ as a vector space for any $k\geq 2$. The isomorphism is given by the map \[\begin{array}{cccc}
\varphi_k: & V^{\otimes k}/\mathcal{I}_k &\rightarrow& \bigwedge_\sigma^k V,\\[3pt]
&x+\mathcal{I}_k&\mapsto& A^{(k)}(x).
\end{array}\]Indeed, by the above lemma, we have $A^{(k)}(\mathcal{I}_k)=0$. So the map is well-defined. For any $y\in V^{\otimes l}$, \begin{align*}
\varphi_{k+l}\big((x+\mathcal{I})(y+\mathcal{I})\big)&=\varphi_{k+l}(x\otimes y+\mathcal{I})\\[3pt]
&=A^{(k+l)}(x\otimes y)\\[3pt]
&=A^{(k)}(x)\shuffle_\sigma A^{(l)}( y)\\[3pt]
&=\varphi_{k}(x+\mathcal{I})\shuffle_\sigma \varphi_{l}(y+\mathcal{I}).
\end{align*}\end{proof}

For $m,n\geq 1$, we define $\Delta_{m,n}:\bigwedge_\sigma^{m+n}V\rightarrow \bigwedge_\sigma^{m}V\otimes \bigwedge_\sigma^{n}V$ by \begin{equation*}\label{Coproduct}\Delta_{m,n}A^{(m+n)}=(A^{(m)}\otimes A^{(n)})\shuffle_{m,n}^-.\end{equation*} We have $$(\Delta_{m,n}\otimes \mathrm{id}_{\bigwedge_\sigma^p V})\Delta_{m+n,p}=(\mathrm{id}_{\bigwedge_\sigma^m V}\otimes \Delta_{n,p})\Delta_{m,n+p},$$for any $m,n,p\geq 1$. It is again a direct consequence of the fact that both of $(\mathfrak{S}_{m+n,p}, \mathfrak{S}_{m,n}\times 1_{\mathfrak{S}_{p}})$ and $(\mathfrak{S}_{m,n+p}, 1_{\mathfrak{S}_{m}}\times\mathfrak{S}_{n,p})$ are reduced decompositions of $\mathfrak{S}_{m,n,p}$. But here we need to take the inverse of both sides of the decompositions. For convenience, we define $\Delta_{0,0}(1)=1\otimes 1$, $$\Delta_{0,m}A^{(m)}(x)=1\otimes A^{(m)}(x),$$ and $$\Delta_{m,0}A^{(m)}(x)= A^{(m)}(x)\otimes 1,$$ for any $x\in V^{\otimes m}$. Therefore $\bigwedge_\sigma V$ is a coalgebra equipped with the coproduct $\Delta$ given by $$\Delta|_{V^{\otimes n}}=\sum_{k=0}^n\Delta_{k,n-k}$$ and the counit given by the projection onto $\mathbb{C}$. Furthermore, it is a twisted bialgebra in the following sense. Define $\beta:T(V)\otimes T(V)\rightarrow
T(V)\otimes T(V)$ by requiring that its restriction on $V^{\otimes m}\otimes V^{\otimes n}$,
denoted by $\beta_{m,n}$, is $\sigma_{\chi_{m,n}}$ , where
\[\chi_{m,n}=\left(\begin{array}{cccccccc}
1&2&\cdots&m&m+1&m+2&\cdots & m+n\\
n+1&n+2&\cdots&n+m&1& 2 &\cdots & n
\end{array}\right)\in \mathfrak{S}_{m+n},\] for any $m,n\geq 1$. The maps
$\beta_{0,m}$ and $\beta_{m,0}$ are defined to be the identity map of $V^{\otimes m}$. Then one has $$\Delta\shuffle_\sigma=(\shuffle_\sigma\otimes \shuffle_\sigma)(\mathrm{id}_{\bigwedge_\sigma V}\otimes \beta\otimes\mathrm{id}_{\bigwedge_\sigma V})(\Delta\otimes\Delta).$$ But we will not use this fact in our paper. One can easily see the following commutativity relation \begin{equation*}\label{Commutativity relation}\shuffle_\sigma \beta_{m,n}(A^{(m)}\otimes A^{(n)})=\shuffle_\sigma(A^{(n)}\otimes A^{(m)})\beta_{m,n}.\end{equation*}

\begin{example}[Quantum exterior algebra of type $A_{N-1}$]\label{Quantum exterior algebra}Let $\sigma$ be the Jimbo braiding given by (\ref{Jimbo braiding}). By Proposition \ref{Quotient version of QEA}, one can show that $\bigwedge_\sigma V$ is the unital associative algebra generated by $v_1,\ldots,v_N$ subject to the relations\begin{equation*}\begin{split}
v_i\wedge v_i&=0,\text{ for any }i,\\[3pt]
v_j\wedge v_i&=-q^{-1}v_i\wedge v_j, \ i<j.
\end{split}
\end{equation*}Here we use the symbol $\wedge$ to denote the product in $\bigwedge_\sigma V$ instead of $\shuffle_\sigma$. We will adopt this convention through this paper for Jimbo's braiding. It is easy to see that for any $\sigma\in \mathfrak{S}_{k}$ and $1\leq i_1<\cdots<i_k\leq N$ with $1\leq k\leq N$, $$v_{i_{w(1)}}\wedge\cdots\wedge v_{i_{w(k)}}=(-q)^{-l(w)}v_{i_1}\wedge\cdots\wedge v_{i_k}.$$Hence the set $$\{v_{i_1}\wedge\cdots\wedge v_{i_k}|1\leq i_1<\cdots<i_k\leq N\}$$forms a basis of $\bigwedge_\sigma^k V$. It implies that $\dim \bigwedge_\sigma^N V=1$. A direct verification shows that for any $1\leq i_1<\cdots<i_{m+n}\leq N$,
\begin{align*}
\lefteqn{\Delta_{m,n}(v_{i_1}\wedge\cdots\wedge v_{i_{m+n}})}\\[3pt]
&=\frac{(m)_{q^{-2}}!(n)_{q^{-2}}!}{(m+n)_{q^{-2}}!}\sum_{w\in \mathfrak{S}_{m,n}}(-q)^{-l(w)}v_{i_{w(1)}}\wedge \cdots\wedge v_{i_{w(m)}}\otimes v_{i_{w(m+1)}}\wedge \cdots\wedge v_{i_{w(m+n)}}.
\end{align*} \end{example}

\begin{example}\label{Quantum exterior algebra of type C}Let $N=4$. We define a braiding $\sigma$ as follows: \begin{align*}
\sigma(v_i\otimes v_i)&= q v_i\otimes v_i,\ i=1,2,3,4,\\[3pt]
\sigma(v_i\otimes v_j)&= v_j\otimes v_i,\ (i,j)=(2,1), (3,1), (4,2), (4,3),\\[3pt]
\sigma(v_i\otimes v_j)&= v_j\otimes v_i+(q-q^{-1})v_i\otimes v_j, \ (i,j)=(1,2),(1,3),(2,4), (3,4),\\[3pt]
\sigma(v_1\otimes v_4)&=q^{-1}v_4\otimes v_1+(1+q^{-4})(q-q^{-1})v_1\otimes v_4,\\[3pt]
&\ \ \ +(q^{-2}-q^{-4})v_2\otimes v_3-(1-q^{-2})v_3\otimes v_2.\\[3pt]
\sigma(v_2\otimes v_3)&=q^{-1}v_3\otimes v_2+(q-q^{-3})v_2\otimes v_3+(q^{-2}-q^{-4})v_1\otimes v_4,\\[3pt]
\sigma(v_3\otimes v_2)&=q^{-1}v_2\otimes v_3-(1-q^{-2})v_1\otimes v_4,\\[3pt]
\sigma(v_4\otimes v_1)&=q^{-1}v_1\otimes v_4.
\end{align*}This is a braiding of type $C$ (see \cite{RTF}). Again we use $\wedge$ to denote the quantum shuffle product as the preceding example. Then one can verify directly that \begin{align*}
v_i\wedge v_i&=0,\ i=1,2,3,4,\\[3pt]
v_i\wedge v_j&=-q^{-1}v_j\wedge v_i, \ i<j,\\[3pt]
v_3\wedge v_2&=-q^2v_2\wedge v_3,\\[3pt]
v_4\wedge v_1&=-q^2v_1\wedge v_4+(q-q^3)v_2\wedge v_3,
\end{align*}and there is no other relations among the generators. It is easy to check that the set $\{v_{i_1}\wedge \cdots \wedge v_{i_k}|1\leq k\leq 4, 1\leq i_1<\cdots<i_k\leq 4\}$, together with 1, form a linear basis of $\bigwedge_\sigma V$.

The coproduct here is much more complicated than that of type $A$. In fact there is no a uniform formula of $\Delta$ as in type $A$. For example, \begin{align*}
(3)_{q^{-2}}\Delta_{1,2}(v_1\wedge v_2\wedge v_3)&=(1-q^{-4}+q^{-6})v_1\otimes v_2\wedge v_3-q^{-5}v_2\otimes v_1\wedge v_3\\[3pt]
&\ \ \ +q^{-3}v_3\otimes v_1\wedge v_2+(q^{-2}-q^{-4})v_1\otimes v_3\wedge v_2\\[3pt]
&\ \ \ +(q^{-3}-q^{-5})v_1\otimes v_1\wedge v_4.
\end{align*}\end{example}

\section{Comodule structures}

Recall that $V$ is an $N$-dimensional vector space with a specified basis $\{v_1,\ldots,v_N\}$. Let $\mathfrak{A}$ be a bialgebra. A linear map $\varrho: V\rightarrow V\otimes \mathfrak{A}$ is said to be \emph{compatible} with a braiding $\sigma$ on $V$ if \begin{equation}\label{Compatibility}\varrho^{\times 2}\sigma=(\sigma\otimes \mathrm{id}_{\mathfrak{A}})\varrho^{\times 2},\end{equation} where $\varrho^{\times 2}$ is given by (\ref{k-fold of cross product}). Set $$\sigma(v_i\otimes v_j)=\sum_{1\leq k,l\leq N}\sigma_{ij}^{kl}v_k\otimes v_l,$$ with some $\sigma_{ij}^{kl}\in \mathbb{C}$. If $\varrho(v_j)=\sum_{i=1}^Nv_i\otimes a^i_j$ for some $a^i_j\in \mathfrak{A}$, then $\varrho$ is compatible with $\sigma$ if and only if \begin{equation*}\sum_{1\leq k,l\leq N}a^m_ka^n_l\sigma_{ij}^{kl}=\sum_{1\leq k,l\leq N}\sigma_{kl}^{mn}a^k_ia^l_j,\end{equation*} for any $1\leq i,j,m,n\leq N$. If $\sigma$ is the Jimbo braiding, the matrix $(a^i_j)$ is called a $q$-matrix in the literature. Denote by $\mathcal{E}^1(\mathfrak{A})^\sigma$ the set of all linear maps from $V$ to $V\otimes \mathfrak{A}$ which are compatible with $\sigma$.

If $\varrho$ is a right $\mathfrak{A}$-comodule structure on $V$, then $\varrho\in \mathcal{E}^1(\mathfrak{A})^\sigma$ means that $\sigma$ is a comodule map.

\begin{lemma}Suppose that $\varrho\in \mathcal{E}^1(\mathfrak{A})^\sigma$, Then, for any $k>1$ and $1\leq i<k$, we have $$\varrho^{\times k}\sigma_i=(\sigma_i\otimes \mathrm{id}_{\mathfrak{A}})\varrho^{\times k}.$$ \end{lemma}
\begin{proof}For any $x\in V^{\otimes (i-1)}$, $y\in V^{\otimes 2}$, and $z\in V^{\otimes (k-i-1)}$, by (\ref{Multiplicative property}) we have\begin{align*}
\varrho^{\times k}\sigma_i(x\otimes y\otimes z)&=\varrho^{\times k}(x\otimes \sigma (y)\otimes z)\\[3pt]
&=\big(\varrho^{\times (i-1)}(x)\big)\diamond\big(\varrho^{\times 2}\sigma (y)\big)\diamond\big(\varrho^{\times (k-i-1)}( z)\big)\\[3pt]
&=\big(\varrho^{\times (i-1)}(x)\big)\diamond\big((\sigma \otimes \mathrm{id}_{\mathfrak{A}})\varrho^{\times 2} (y)\big)\diamond\big(\varrho^{\times (k-i-1)}( z)\big)\\[3pt]
&=(\sigma_i\otimes \mathrm{id}_{\mathfrak{A}})\Big(\big(\varrho^{\times (i-1)}(x)\big)\diamond\big(\varrho^{\times 2} (y)\big)\diamond\big(\varrho^{\times (k-i-1)}( z)\big)\Big)\\[3pt]
&=(\sigma_i\otimes \mathrm{id}_{\mathfrak{A}})\varrho^{\times k}(x\otimes y\otimes z).
\end{align*}\end{proof}

It follows from this lemma immediately that\begin{equation}\label{rho and A}\varrho^{\times k}A^{(k)}=(A^{(k)}\otimes \mathrm{id}_{\mathfrak{A}})\varrho^{\times k}\end{equation} for any  $\varrho\in \mathcal{E}^1(\mathfrak{A})^\sigma$ and each $k>1$. Recall that if $\varrho$ is moreover an $\mathfrak{A}$-comodule structure on $V$, then $T(V)$ is an $\mathfrak{A}$-comodule as well, whose comodule structure is given by $\sum_k \varrho^{\times k}$. Furthermore, $T(V)$ is a comodule algebra. Due to (\ref{rho and A}), $\bigwedge_\sigma V$ is a subcomodule of $T(V)$. Define an associative product $\cdot$ on $\bigwedge_\sigma V\otimes \mathfrak{A}$ by \begin{equation}\label{Tensor product}\cdot=(\shuffle_\sigma\otimes \mathfrak{m})(\mathrm{id}_{\bigwedge_\sigma V}\otimes \tau\otimes \mathrm{id}_{\mathfrak{A}}),\end{equation} where $\mathfrak{m}$ is the multiplication map of $\mathfrak{A}$ and $\tau$ is the flip map. Then we have

\begin{proposition}\label{Comodule structure for QEA}The quantum exterior algebra $\bigwedge_\sigma V$ is a right $\mathfrak{A}$-comodule algebra.\end{proposition}
\begin{proof}For any $x\in V^{\otimes m}$ and $y\in V^{\otimes n}$, we have\begin{align*}
\varrho^{\times (m+n)}\big(A^{(m)}(x)\shuffle_\sigma A^{(n)}(y)\big)&=\varrho^{\times (m+n)}\big(A^{(m+n)}(x\otimes y)\big)\\[3pt]
&=(A^{(m+n)}\otimes \mathrm{id}_{\mathfrak{A}})\varrho^{\times (m+n)}(x\otimes y)\\[3pt]
&=\sum A^{(m+n)}(x\otimes y)_{(0)}\otimes (x\otimes y)_{(1)}\\[3pt]
&=\sum A^{(m+n)}(x_{(0)}\otimes y_{(0)})\otimes x_{(1)} y_{(1)} \\[3pt]
&=\sum \big(A^{(m)}(x_{(0)})\shuffle_\sigma A^{(m)}( y_{(0)})\big)\otimes x_{(1)} y_{(1)}\\[3pt]
&=\big(\varrho^{\times m}A^{(m)}(x)\big)\cdot \big(\varrho^{\times n}A^{(n)}(y)\big).
\end{align*}\end{proof}

Let us consider an important example. Recall that the Faddeev-Reshetikhin-Takhtajan algebra (abbreviated by FRT algebra) $\mathcal{A}(\sigma)$ associated to $\sigma$ is the unital associative algebra over $\mathbb{C}$ generated by $\{T^i_j\}_{1\leq i,j\leq N}$ subject to the relations \begin{equation}\label{RTT relations}\sum_{1\leq k,l\leq N}\sigma_{ij}^{kl}T^m_kT^n_l=\sum_{1\leq k,l\leq N}T^k_iT^l_j\sigma_{kl}^{mn},\end{equation} for any $1\leq i,j,m,n\leq N$. In addition, $\mathcal{A}(\sigma)$ is a bialgebra with the coproduct $$\Delta_{\mathcal{A}(\sigma)} (T^i_j)=\sum_{k=1}^NT^i_k\otimes T^k_j,$$ and counit $$\varepsilon_{\mathcal{A}(\sigma)}(T^i_j)=\delta^i_j,$$ where $\delta^i_j$ is the usual Kronecker symbol. It is well-known that $V$ possesses a canonical right $\mathcal{A}(\sigma)$-comodule structure with the comodule structure map $\rho$ defined by \begin{equation}\label{Comodule map for FRT algebra}\rho(v_i)=\sum_{j=1}^Nv_j\otimes T^j_i,\end{equation} for each $i=1,2,\ldots,N$. Furthermore, $\rho^{\times 2}\sigma=(\sigma\otimes \mathrm{id}_{\mathfrak{A}})\rho^{\times 2}$. Thus

\begin{corollary}\label{FRT-comodule}The quantum exterior algebra $\bigwedge_\sigma V$ is a right $\mathcal{A}(\sigma)$-comodule algebra whose comodule structure map on the component $\bigwedge_\sigma V$ is given by  \begin{equation*}\label{Comodule map}\rho^{\times k}\big(A^{(k)}(v_{i_1}\otimes \cdots\otimes v_{i_k})\big)=\sum_{1\leq j_1,\ldots,j_k\leq N}A^{(k)}(v_{j_1}\otimes \cdots\otimes v_{j_k})\otimes T_{i_1}^{j_1}\cdots T_{i_k}^{j_k}.\end{equation*}\end{corollary}
\begin{proof}Apply the preceding proposition to $\rho:V\rightarrow V\otimes \mathcal{A}(\sigma)$.\end{proof}

\begin{remark}The above corollary has also been obtained in \cite{FG} (see Proposition 1.7 therein).\end{remark}

The following special kind of braidings is of particular interest in quantum groups. Let $(a_{ij})_{1\leq i,j\leq N}$ be a symmetrizable generalized Cartan matrix with symmetrization $\mathrm{diag}(d_1,\ldots, d_N)$. Thus the matrix $(d_ia_{ij})_{1\leq i,j\leq N}$ is symmetric. The positive part $U_q^+$ of the quantum enveloping algebra associated to $(a_{ij})_{1\leq i,j\leq N}$ is the unital associative algebra generated by $e_1,\ldots,e_N$ subject to the famous quantum Serre relations $$\sum_{k=0}^{1-a_{ij}}(-1)^k\left[\begin{array}{c}
1-a_{ij}\\
k
\end{array}\right]_ie^k_ie_je^{1-a_{ij}-k}_i=0$$ for any $1\leq i\neq j\leq N$. Here the quantum integers and quantum binomial coefficients are defined respectively by $$[n]_i=\frac{q^{d_in}-q^{-d_in}}{q^{d_i}-q^{-d_i}},$$ and $$\left[\begin{array}{c}
m\\
n
\end{array}\right]_i=\frac{[m]_i[m-1]_i\cdots [m-n+1]_i}{[n]_i[n-1]_i\cdots [1]_i},$$ where $0\leq n\leq m$. Define $$\sigma (v_i\otimes v_j)=-q^{d_ia_{ij}+1}v_j\otimes v_i$$ It was showed by Rosso (see Theorem 15 in \cite{Ro}) that $\bigwedge_\sigma V$ is isomorphic to $U_q^+$ as an algebra. More precisely, the isomorphism is given by $v_i\mapsto e_i$. Combining this fact and the above corollary, we have

\begin{theorem}The algebra $U_q^+$ is a right $\mathcal{A}(\sigma)$-comodule algebra whose comodule map is induced by the map $e_i\mapsto \sum_{j=1}^N e_j\otimes T^j_i$.\end{theorem}

\begin{remark}Keeping the assumptions above, the relations for the generators $T^i_j$ are $$T^i_kT^j_l=q^{d_ka_{kl}-d_ia_{ij}}T^j_l T^i_k,$$for any $1\leq i,j,k,l\leq N$. For example, consider the Cartan matrix of type $A_2$: $$\left(\begin{array}{cc}
2&-1\\
-1&2
\end{array}\right).$$The matrix of the corresponding braiding with respect to the ordered basis $\{v_1\otimes v_1,v_1\otimes v_2,v_2\otimes v_1,v_2\otimes v_2\}$ is $$-q\left(\begin{array}{cccc}
q^2&0&0&0\\
0&0&q^{-1}&0\\
0&q^{-1}&0&0\\
0&0&0&q^2
\end{array}\right).$$ The corresponding FRT algebra is generated by four generators $a,b,c,d$ subject to the relations $$ab=ba=0,ac=ca=0, bd=db=0, cd=dc=0,$$ $$bc=cb, ad=da.$$ \end{remark}

We conclude this section by a useful observation about the product $\cdot$ given by (\ref{Tensor product}). Let $\bigwedge_\sigma V$ be the quantum exterior algebra as in Example \ref{Quantum exterior algebra} and $\rho:V\rightarrow V\otimes \mathcal{A}(\sigma)$ be the canonical $\mathcal{A}(\sigma)$-comodule structure map. It is easy to verify that \begin{equation*}\begin{split}
\rho(v_i)\cdot\rho( v_i)&=0,\text{ for any }i,\\[3pt]
\rho(v_j)\cdot\rho( v_i)&=-q^{-1}\rho(v_i)\cdot \rho(v_j), \ i<j.
\end{split}
\end{equation*}It means that the subalgebra of $\bigwedge_\sigma V\otimes \mathcal{A}(\sigma)$ generated by $\rho(V)$ shares the same multiplication rule with $\bigwedge_\sigma V$. This fact simplifies many computations, for example, Example \ref{Quantum minor determinant} below. Let us generalize this fact to a more general setting. We assume that $\sigma$ is of Hecke type and $\mathfrak{A}$ is a bialgebra. We claim that if $\varrho:V\rightarrow V\otimes  \mathfrak{A}$ is a comodule structure map satisfying (\ref{Compatibility}), then the subalgebra of $\bigwedge_\sigma V\otimes \mathfrak{A}$ generated by $\varrho(V)$ shares the same multiplication rule with $\bigwedge_\sigma V$. According to Proposition \ref{Quotient version of QEA}, the multiplication relations of $\bigwedge_\sigma V$ among elements in $V$ come from $\mathrm{Ker}(\sigma-qI_2)$.

\begin{proposition}Under the assumptions above, we have $$\cdot(\varrho\otimes \varrho)\big(\mathrm{Ker}(\sigma-qI_2)\big)=0.$$\end{proposition}
\begin{proof}For any $\sum_i x_i\otimes y_i\in \mathrm{Ker}(\sigma-qI_2)$, \begin{align*}
\cdot(\varrho\otimes \varrho)\Bigg(\sum_i x_i\otimes y_i\Bigg)&=\sum_i\Bigg(\sum_{(x_i)} x_{i,(0)}\otimes x_{i,(1)}\Bigg)\cdot \Bigg(\sum_{(y_i)} y_{i,(0)}\otimes y_{i,(1)}\Bigg)\\[3pt]
&=\sum_i\sum_{(x_i)}\sum_{(y_i)}A^{(2)}( x_{i,(0)}\otimes y_{i,(0)})\otimes x_{i,(1)}y_{i,(1)}\\[3pt]
&=\sum_i\sum_{(x_i\otimes y_i)}A^{(2)}( (x_i\otimes y_i)_{(0)})\otimes (x_i\otimes y_i)_{(1)}\\[3pt]
&=(A^{(2)}\otimes \mathrm{id}_{\mathfrak{A}})\varrho^{\times 2}\Bigg(\sum_i x_i\otimes y_i\Bigg)\\[3pt]
&=\varrho^{\times 2}A^{(2)}\Bigg(\sum_i x_i\otimes y_i\Bigg)\\[3pt]
&=0,
\end{align*}where the second to the last equality follows from (\ref{rho and A}) and the last equality from Lemma \ref{Lemma for kernel}.\end{proof}

\section{Convolution products}
In this section, $\sigma$ is always a braiding on $V$. As before, we fix a basis $\{v_1,\ldots,v_N\}$ of $V$.
Let $\mathfrak{A}$ be an associative algebra, and $\cdot$ the product on $\bigwedge_\sigma V\otimes \mathfrak{A}$ defined by (\ref{Tensor product}).

For any nonnegative integer $k$, we denote $\mathcal{E}^k(\mathfrak{A})=\mathrm{Hom}_\mathbb{C}\big(\bigwedge_\sigma^kV,\bigwedge_\sigma^kV\otimes \mathfrak{A}\big)$. Obviously $\mathcal{E}^k(\mathfrak{A})$ is isomorphic to $\mathrm{End}(\bigwedge_\sigma^kV)\otimes \mathfrak{A}$ as vector spaces. So we will view $\mathrm{End}(\bigwedge_\sigma^kV)$ as the subsapce $\mathrm{End}(\bigwedge_\sigma^kV)\otimes 1_{\mathfrak{A}}$ of $\mathcal{E}^k(\mathfrak{A})$ in the sequel. The space $\mathcal{E}^k(\mathfrak{A})$ possesses an $\mathfrak{A}$-module structure as follows: for any $F\in \mathcal{E}^k(\mathfrak{A})$, $\alpha\in \bigwedge_\sigma^kV$, and $a\in \mathfrak{A}$, $$(a.F)(\alpha)=(1\otimes a)\cdot F(\alpha).$$ It is not hard to see that $\mathcal{E}^k(\mathfrak{A})$ is a free $\mathfrak{A}$-module of rank $\dim \bigwedge_\sigma^kV$.

\begin{definition}For any $F\in \mathcal{E}^m(\mathfrak{A})$ and $G\in \mathcal{E}^n(\mathfrak{A})$, the \emph{convolution product} $F\ast G\in \mathcal{E}^{m+n}(\mathfrak{A})$ of $F$ and $G$ is defined to be $$F\ast G=\cdot(F\otimes G)\Delta_{m,n}.$$\end{definition}

We can extend this product to $\mathcal{E}(\mathfrak{A})=\bigoplus_{k\geq 0}\mathcal{E}^k(\mathfrak{A})$ as follows. For any $\mathbf{F}=(F_0,F_1,\ldots),\mathbf{G}=(G_0,G_1,\ldots)\in \mathcal{E}(\mathfrak{A})$ with $F_i,G_i\in \mathcal{E}^i(\mathfrak{A})$, we define $\mathbf{F}\ast \mathbf{G} =(H_0,H_1,\ldots),$ where $$H_k=\sum_{i=0}^kF_i\ast G_{k-i}.$$

\begin{proposition}Together with the convolution product, $\mathcal{E}(\mathfrak{A})$ is a unital associative algebra whose unit is $I_0$.\end{proposition}
\begin{proof}For the associativity, it suffices to verify that $(F\ast G)\ast H=F\ast ( G\ast H)$ for any $F\in \mathcal{E}^m(\mathfrak{A})$, $G\in \mathcal{E}^n(\mathfrak{A})$, and $H\in \mathcal{E}^p(\mathfrak{A})$. Notice that \begin{align*}
(F\ast G)\ast H&=\cdot\big((F\ast G)\otimes H\big)\Delta_{m+n,p}\\[3pt]
&=\cdot(\cdot \otimes \mathrm{id}_{\bigwedge_\sigma^p\otimes\mathfrak{A}})(F\otimes G\otimes H)(\Delta_{m,n}\otimes \mathrm{id}_{\bigwedge^p_\sigma V})\Delta_{m+n,p}\\[3pt]
&=\cdot(\mathrm{id}_{\bigwedge_\sigma^m\otimes\mathfrak{A}} \otimes\cdot)(F\otimes G\otimes H)(\mathrm{id}_{\bigwedge^m_\sigma V}\otimes\Delta_{n,p})\Delta_{m,n+p}\\[3pt]
&=F\ast ( G\ast H),
\end{align*}where the third equality follows from the associativity of $\cdot$ and the coassociativity of $\Delta$.

Obviously $I_0=(I_0,0,0,\ldots)$ is the unit with respect to $\ast$. \end{proof}

\begin{remark}(1) In general, the product $\ast$ is not commutative.

(2) Suppose that $\sigma$ is the usual flip map, $q=1$, and $\mathfrak{A}=\mathbb{C}$. Then the product $\ast$ of two maps in $\mathcal{E}(\mathfrak{A})$ is just the usual exterior product of two endomorphisms. Hence $\ast$ is a quantization of the usual exterior product of endomorphisms.\end{remark}

\begin{proposition}\label{k-fold of quantum shuffle product}If $\sigma$ is of Hecke type and $\varrho\in \mathcal{E}^1(\mathfrak{A})^\sigma$, then \begin{equation*}\varrho^{\ast k}=\varrho^{\times k}\end{equation*}for each $k\geq 1$. \end{proposition}
\begin{proof}We use induction on $k$. The case $k=1$ is trivial. Assume that the identity holds for $k\geq 1$. Then\begin{align*}
\varrho^{\ast (k+1)}A^{(k+1)}&=(\varrho^{\ast k}\ast \varrho)A^{(k+1)}\\[3pt]
&=\cdot (\varrho^{\ast k}\otimes \varrho)\Delta_{k,1}A^{(k+1)}\\[3pt]
&=\cdot (\varrho^{\ast k}\otimes \varrho)(A^{(k)}\otimes A^{(1)})\shuffle_{k,1}^-\\[3pt]
&=\cdot (\varrho^{\ast k}A^{(k)}\otimes \varrho A^{(1)})\shuffle_{k,1}^-\\[3pt]
&=\cdot (\varrho^{\times k}A^{(k)}\otimes \varrho  A^{(1)})\shuffle_{k,1}^-\\[3pt]
&=\varrho^{\times (k+1)} A^{(k+1)}\shuffle_{k,1}^-\\[3pt]
&=\varrho^{\times (k+1)} A^{(k+1)}\Big(\frac{(k)_{q^{-2}}!}{(k+1)_{q^{-2}}!}\sum_{i=1}^{k+1}(-q)^{-(k+1-i)}\sigma_k\sigma_{k-1}\cdots\sigma_i\Big)\\[3pt]
&=\varrho^{\times (k+1)}A^{(k+1)},
\end{align*}where the fifth equality follows from the inductive hypothesis, the sixth one from Proposition \ref{Comodule structure for QEA}, and the last one from (\ref{Gurevich1}).\end{proof}

We apply the above proposition to the case $\mathfrak{A}=\mathcal{A}(\sigma)$ and $\varrho=\rho$ defined by (\ref{Comodule map for FRT algebra}). Let $\{\omega_{k,1}, \ldots, \omega_{k,d_k}\}$ be a basis of $\bigwedge_\sigma^kV$ where $d_k=\dim \bigwedge_\sigma^kV$. For any $F\in \mathcal{E}^k(\mathcal{A}(\sigma))$, there exist unique elements $a_i^j\in \mathcal{A}(\sigma)$ such that $F(\omega_{k,i})=\sum_{j=1}^{d_k}\omega_{k,j}\otimes a_i^j$ for each $i$. We define the \emph{character} of $F$ to be the element $\mathrm{ch} F=\sum_i a_i^i$. For any $1\leq i_1,\ldots,i_k\leq N$, we denote $$A^{(k)}(v_{i_1}\otimes \cdots\otimes v_{i_k})=\sum_{1\leq j_1,\ldots,j_k\leq N}(A^{(k)})_{i_1 \ldots i_k}^{j_1 \ldots j_k}v_{j_1}\otimes \cdots\otimes v_{j_k},$$with $(A^{(k)})_{i_1 \ldots i_k}^{j_1 \ldots j_k}\in \mathbb{C}$.

\begin{proposition}We have $$\mathrm{ch}\rho^{\ast k}=\sum_{1\leq j_1,\ldots,j_k\leq N} (A^{(k)})_{j_1 \ldots j_k}^{i_1 \ldots i_k}T^{j_1}_{i_1}\cdots T^{j_k}_{i_k}.$$\end{proposition}
\begin{proof}By the relation (\ref{RTT relations}) of the FRT algebra $\mathcal{A}(\sigma)$, we have \begin{equation*} \sum_{1\leq l_1,\ldots,l_k\leq N}(A^{(k)})^{i_1 \ldots i_k}_{l_1 \ldots l_k}T^{l_1}_{j_1}\cdots T^{l_k}_{j_k}=\sum_{1\leq l_1,\ldots,l_k\leq N}T^{i_1}_{l_1}\cdots T^{i_k}_{l_k}(A^{(k)})^{l_1 \ldots l_k}_{j_1 \ldots j_k},\end{equation*} for any $1\leq i_1,\ldots,i_k,j_1,\dots,j_k\leq N$. Combining this and Proposition \ref{k-fold of quantum shuffle product}, we have\begin{align*}
\lefteqn{\rho^{\ast k}A^{(k)}(v_{i_1}\otimes \cdots\otimes v_{i_k})}\\[3pt]
&=\rho^{\times k}\Big(\sum_{1\leq j_1,\ldots,j_k\leq N}(A^{(k)})_{i_1 \ldots i_k}^{j_1 \ldots j_k}v_{j_1}\otimes \cdots\otimes v_{j_k}\Big)\\[3pt]
&=\sum_{1\leq j_1,\ldots,j_k\leq N}(A^{(k)})_{i_1 \ldots i_k}^{j_1 \ldots j_k}\rho^{\times k} (v_{j_1}\otimes \cdots\otimes v_{j_k})\\[3pt]
&=\sum_{1\leq j_1,\ldots,j_k\leq N}(A^{(k)})_{i_1 \ldots i_k}^{j_1 \ldots j_k}\sum_{1\leq l_1,\ldots,l_k\leq N}v_{l_1}\otimes \cdots\otimes v_{l_k}\otimes T^{l_1}_{j_1}\cdots T^{l_k}_{j_k}\\[3pt]
&=\sum_{1\leq l_1,\ldots,l_k\leq N}v_{l_1}\otimes \cdots\otimes v_{l_k}\otimes\sum_{1\leq j_1,\ldots,j_k\leq N} T^{l_1}_{j_1}\cdots T^{l_k}_{j_k}(A^{(k)})_{i_1 \ldots i_k}^{j_1 \ldots j_k}\\[3pt]
&=\sum_{1\leq l_1,\ldots,l_k\leq N}v_{l_1}\otimes \cdots\otimes v_{l_k}\otimes\sum_{1\leq j_1,\ldots,j_k\leq N} (A^{(k)})_{j_1 \ldots j_k}^{l_1 \ldots l_k}T^{j_1}_{i_1}\cdots T^{j_k}_{i_k}.
\end{align*}

Without loss of generality, we may assume that $A^{(k)}(\omega_{k,i})=\omega_{k,i}$ for $1\leq i\leq d_k$ since $A^{(k)}$ is a projection by (\ref{Gurevich2}). Extend these $\omega_{k,i}$'s to a basis of $V^{\otimes k}$. Under this basis, it is easy to see that $$\mathrm{ch}\rho^{\ast k}=\sum_{1\leq j_1,\ldots,j_k\leq N} (A^{(k)})_{j_1 \ldots j_k}^{i_1 \ldots i_k}T^{j_1}_{i_1}\cdots T^{j_k}_{i_k},$$as desired.\end{proof}

\begin{remark}The elements $\mathrm{ch}\rho^{\ast k}$'s are just a $q$-analogue of the elementary symmetric functions in $\mathcal{A}(\sigma)$. They are used to establish the quantum Cayley-Hamilton theorem \cite{IOPS}, and are extended to some kind of quantum matrix algebras related to compatible braidings \cite{IOP}.\end{remark}

\section{Quantum determinants}

Let $d_k$ be the dimension of the vector space $\bigwedge_\sigma^k V$, and $\{\omega_{k,1}, \ldots, \omega_{k,d_k}\}$ a basis of $\bigwedge_\sigma^kV$ for each positive integer $k$. Assume that $\mathfrak{A}$ is an associative algebra.

\begin{definition}Given $F\in \mathcal{E}^1(\mathfrak{A})$ and $k\geq 1$, the elements $\xi(F,k)_{i}^j\in \mathfrak{A}$ determined uniquely by $$F^{\ast k}(\omega_{k,i})=\sum_{j=1}^{d_k}\omega_{k,j}\otimes \xi(F,k)_{i}^j$$ are called the \emph{quantum $k$-minor determinants} of $F$.\end{definition}

\begin{remark}Suppose that $\mathfrak{A}$ is moreover a bialgebra with coproduct $\Delta_\mathfrak{A}$ and counit $\varepsilon_\mathfrak{A}$. If $F^{\ast k}$ is a comodule structure map, then we have immediately \begin{equation*}\Delta_{\mathfrak{A}}(\xi(F,k)_{j}^i)=\sum_{n=1}^{d_k}\xi(F,k)_{n}^i\otimes \xi(F,k)_{j}^n,\end{equation*}and
\begin{equation*}\varepsilon_{\mathfrak{A}}(\xi(F,k)_{j}^i)=\delta^i_j.\end{equation*}For example, if $\sigma$ is of Hecke type, then comdoule structure maps satisfying (\ref{Compatibility}) will fulfill this condition.\end{remark}

A braiding $\sigma$ is said to be of \emph{finite rank} if there exists an integer $M\geq 1$ such that $\dim \bigwedge_\sigma^M V=1$ and $\dim \bigwedge_\sigma^k V=0$ for any $k>M$. The integer $M$ is called the \emph{rank} of $\sigma$. We have seen that Jimbo's braidings are of finite rank. Another example is the braiding given in Example \ref{Quantum exterior algebra of type C} whose rank is 4.

\begin{definition}Let $\sigma$ be a braiding of rank $M$. For any $F\in \mathcal{E}^1(\mathfrak{A})$, the quantum $M$-minor determinant of $F$ is called the \emph{quantum determinant} of $F$, and is denoted by $\det_q F$.\end{definition}

\begin{example}\label{Quantum minor determinant}Let us consider the quantum exterior algebra of type $A_{N-1}$ as in Example \ref{Quantum exterior algebra}. Recall that for $1\leq k\leq N$ the set $\{v_{i_1}\wedge\cdots\wedge v_{i_k}|1\leq i_1<\cdots<i_k\leq N\}$ forms a linear basis of $\bigwedge_\sigma^k V$. Let $\rho$ be the canonical $\mathcal{A}(\sigma)$-comodule structure map.
For  $1\leq j_1<\cdots<j_k\leq N$, we have\begin{align*}
\lefteqn{\rho^{\ast k}(v_{j_1}\wedge \cdots\wedge v_{j_k})}\\[3pt]
&=(\cdot^{k-1}\rho^{\otimes k} \Delta_{(1,\cdots,1)})(v_{j_1}\wedge \cdots\wedge v_{j_k})\\[3pt]
&=(\cdot^{k-1}\rho^{\otimes k})
\Big(\frac{1}{(k)_{q^{-2}}!}\sum_{w\in \mathfrak{S}_{k}}(-q)^{l(w)}v_{j_{w(1)}}\otimes\cdots \otimes v_{j_{w(k)}}\Big)\\[3pt]
&=\frac{1}{(k)_{q^{-2}}!}\sum_{w\in \mathfrak{S}_{k}}(-q)^{l(w)}\rho(v_{j_{w(1)}})\cdot\cdots \cdot \rho(v_{j_{w(k)}})
\\[3pt]
&=\frac{1}{(k)_{q^{-2}}!}\sum_{w\in \mathfrak{S}_{k}}(-q)^{-2l(w)}
\rho(v_{j_1})\cdot\cdots \cdot \rho_1(v_{j_k})
\\[3pt]
&=\rho(v_{j_1})\cdot \cdots \cdot \rho(v_{j_k})
\\[3pt]
&=\sum_{1\leq i_1,\ldots,i_k\leq N}
v_{i_1}\wedge \cdots\wedge  v_{i_k}\otimes T^{i_1}_{j_1}\cdots T^{i_k}_{j_k}\\[3pt]
&=
\sum_{1\leq i_1<\cdots<i_k\leq N}
\sum_{w\in \mathfrak{S}_{k}}(-q)^{-l(w)}v_{i_1}\wedge \cdots\wedge  v_{i_k}\otimes T^{i_{w(1)}}_{j_1}\cdots T^{i_{w(k)}}_{j_k}\\[3pt]
&=
\sum_{1\leq i_1<\cdots<i_k\leq N}
v_{i_1}\wedge \cdots\wedge  v_{i_k}\otimes \sum_{w\in \mathfrak{S}_{k}}(-q)^{-l(w)}T^{i_{w(1)}}_{j_1}\cdots T^{i_{w(k)}}_{j_k}.
\end{align*}
So$$\xi^{i_1,\cdots,i_k}_{j_1,\cdots,j_k}=\xi(\rho,k)^{i_1,\cdots,i_k}_{j_1,\cdots,j_k}=\sum_{w\in \mathfrak{S}_{k}}(-q)^{-l(w)}T^{i_{w(1)}}_{j_1}\cdots T^{i_{w(k)}}_{j_k}.$$ They coincide with those quantum minor determinants introduced in \cite{NYM}.

Since $\dim \bigwedge_\sigma^NV=1$ and $\dim \bigwedge_\sigma^kV=0$ for $k>N$, we have  $$\det\nolimits_q(\rho)=\sum_{w\in \mathfrak{S}_{N}}(-q)^{-l(w)}T^{w(1)}_{1}\cdots T^{w(N)}_{N},$$which is just the classical quantum determinant.
 \end{example}

We are going to generalise the Laplace expansion and the multiplicative formula for determinants to our quantum minor determinants. Let $m,n\geq  1$ and $p=m+n$. Assume that \begin{equation*}\omega_{m,i}\shuffle_\sigma \omega_{n,j}=\sum_{k=1}^{d_p}\mu(m,n)_{i,j}^k\omega_{p,k},\end{equation*}and
\begin{equation*}\Delta_{m,n}(\omega_{p,k})=\sum_{i=1}^{d_m}\sum_{j=1}^{d_n}\lambda(p)_k^{i,j}\omega_{m,i}\otimes \omega_{n,j}.\end{equation*}Then we have the following quantum Laplace expansion formula.

\begin{theorem}Under the assumptions above, for any $F\in \mathcal{E}^1(\mathfrak{A})$, we have \begin{equation*}\xi(F,m+n)^i_j=\sum_{k,r=1}^{d_m}\sum_{l,s=1}^{d_n}\lambda(m+n)_j^{k,l}\mu(m,n)_{r,s}^i \xi(F,m)_k^r  \xi(F,n)_l^s\end{equation*}\end{theorem}
\begin{proof}Notice that \begin{align*}
F^{\ast p}(\omega_{p,j})&=(F^{\ast m}\ast F^{\ast n})(\omega_{p,j})\\
&=\cdot (F^{\ast m}\otimes F^{\ast n})\Delta_{m,n}(\omega_{p,j})\\
&=\cdot (F^{\ast m}\otimes F^{\ast n})\Big(\sum_{k=1}^{d_m}\sum_{l=1}^{d_n}\lambda(p)_j^{k,l}\omega_{m,k}\otimes \omega_{n,l}\Big)\\
&=\sum_{k=1}^{d_m}\sum_{l=1}^{d_n}\lambda(p)_j^{k,l}F^{\ast m}(\omega_{m,k})\cdot F^{\ast n}(\omega_{n,l})\\
&=\sum_{k=1}^{d_m}\sum_{l=1}^{d_n}\lambda(p)_j^{k,l}\Big(\sum_{r=1}^{d_m}\omega_{m,r}\otimes \xi(F,m)_k^r\Big)\cdot \Big(\sum_{s=1}^{d_n}\omega_{n,s}\otimes \xi(F,n)_l^s\Big)\\
&=\sum_{k,r=1}^{d_m}\sum_{l,s=1}^{d_n}\lambda(p)_j^{k,l}\omega_{m,r}\shuffle_\sigma \omega_{n,s}\otimes \xi(F,m)_k^r  \xi(F,n)_l^s\\
&=\sum_{k,r=1}^{d_m}\sum_{l,s=1}^{d_n}\lambda(p)_j^{k,l}\sum_{i=1}^{d_p}\mu(m,n)_{r,s}^i\omega_{p,i}\otimes \xi(F,m)_k^r  \xi(F,n)_l^s.
\end{align*}Comparing the last formula with $F^{\ast p}(\omega_{p,j})=\sum_{i=1}^{d_p}\omega_{p,i}\otimes \xi(F,p)^i_j$, we get the desired result.\end{proof}

\begin{example}Let $\bigwedge_\sigma V$ be the quantum exterior of type $A_{N-1}$. Let $J$ and $K$ be two subsets of $\{1,2,\ldots, N\}$. We write $$l(J;K)=\sharp\{(j,k)|j\in J,k\in K, j>k\}.$$ If $J$ and $K$ have the same cardinality, we denote by $\delta^J_K$ the generalized Kronecker symbol, i.e., $\delta^J_K=1$ if $J=K$, and $\delta^J_K=0$ if $J\neq K$. Apparently, for any three increasing sequences $1\leq i_1<\cdots<i_{m+n}\leq N$, $1\leq j_1<\cdots<j_{m}\leq N$, and $1\leq k_1<\cdots<k_{n}\leq N$, we have\begin{equation*}\mu^{ i_1,\ldots,i_{m+n}}_{j_1,\ldots,j_{m};k_1,\ldots,k_n}=(-q)^{-l(j_1,\ldots,j_m;k_1,\ldots,k_n)}\delta^{i_1,\ldots,i_{m+n}}_{j_1,\ldots,j_m,k_1,\ldots,k_n},\end{equation*}and
\begin{equation*}\lambda_{ i_1,\ldots,i_{m+n}}^{j_1,\ldots,j_{m};k_1,\ldots,k_n}=\frac{(m)_{q^{-2}}!(n)_{q^{-2}}!}{(m+n)_{q^{-2}}!}(-q)^{-l(j_1,\ldots,j_m;k_1,\ldots,k_n)}\delta^{i_1,\ldots,i_{m+n}}_{j_1,\ldots,j_m,k_1,\ldots,k_n}.\end{equation*}

Given $1\leq r\leq N$ and $1\leq i_1<\cdots<i_r\leq N$, we denote by $C(i_1,\ldots,i_r)$ the complement of the set $\{i_1,\ldots,i_r\}$ in $\{1,\ldots,N\}$. Elements in $C(i_1,\ldots,i_r)$ are always arranged in an increasing order. An easy computation shows that $$l(i_1,\ldots,i_r;C(i_1,\ldots,i_r))=i_1+\cdots+i_r-\frac{r(r+1)}{2}.$$This is just the length of the $(r,N-r)$-shuffle $w$ determined by $w(1)=i_1,\ldots,w(N)=i_N$, where $C(i_1,\ldots,i_r)=\{i_{r+1},\ldots,i_N\}$. By the above theorem, we have\begin{align*}
\lefteqn{\xi_{1,\ldots,N}^{1,\ldots,N}}\\
&=\det\nolimits_q(\rho)\\
&=\sum_{\substack{1\leq i_1<\cdots<i_r\leq N,\\ 1\leq j_1<\cdots<j_r\leq N }}\sum_{\substack{1\leq i_{r+1}<\cdots<i_N\leq N,\\ 1\leq j_{r+1}<\cdots<j_N\leq N }}\mu^{ 1,\ldots,N}_{j_1,\ldots,j_{r};j_{r+1},\ldots,j_N}\lambda_{ 1,\ldots,N}^{j_1,\ldots,j_{r};j_{r+1},\ldots,j_N}\xi^{j_1,\ldots,j_r}_{i_1,\ldots,i_r}\xi^{j_{r+1},\ldots,j_N}_{i_{r+1},\ldots,i_N}\\
&=\sum_{\substack{1\leq i_1<\cdots<i_r\leq N,\\ 1\leq j_1<\cdots<j_r\leq N }}(-q)^{-l(i_1,\ldots,i_r;C(i_1,\ldots,i_r))-l(j_1,\ldots,j_r;C(j_1,\ldots,j_r))}\frac{(r)_{q^{-2}}!(N-r)_{q^{-2}}!}{(N)_{q^{-2}}!}\xi^{j_1,\ldots,j_r}_{i_1,\ldots,i_r}\xi^{C(j_1,\ldots,j_r)}_{C(i_{1},\ldots,i_r)}\\
&=\frac{(r)_{q^{-2}}!(N-r)_{q^{-2}}!}{(N)_{q^{-2}}!}\sum_{w,w'\in\mathfrak{S}_{r,N-r}}(-q)^{-l(w)-l(w')}\xi^{w'(1),\ldots,w'(r)}_{w(1),\ldots,w(r)}\xi^{C(w'(1),\ldots,w'(r))}_{C(w(1),\ldots,w(r))}.\end{align*}

This formula for $\xi_{1,\ldots,N}^{1,\ldots,N}$ can be also obtained from the classical $q$-analogue of Laplace expansion (see, e.g., \cite{NYM} or \cite{JZ}) and (\ref{Quantum binomial}). \end{example}

Now we turn to the multiplicative formula for quantum determinants. In the classical case, it asserts that the determinant function commutes with the composition of linear endomorphisms.

\begin{definition}For any $F,G\in \mathcal{E}^k(\mathfrak{A})$, the \emph{composition product} $F\circ G\in \mathcal{E}^k(\mathfrak{A})$ is defined to be $$F\circ G=(\mathrm{id}_{\bigwedge_\sigma^k V}\otimes \mathfrak{m})(F\otimes \mathrm{id}_{\mathfrak{A}})G.$$\end{definition}

\begin{proposition}Together with the composition product, $\mathcal{E}^k(\mathfrak{A})$ is a unital associative algebra.\end{proposition}
\begin{proof}For any $F,G, H\in \mathcal{E}^k(\mathfrak{A})$, we have\begin{align*}
(F\circ G)\circ H&=(\mathrm{id}_{\bigwedge_\sigma^k V}\otimes \mathfrak{m})(F\circ G\otimes \mathrm{id}_{\mathfrak{A}})H\\[3pt]
&=(\mathrm{id}_{\bigwedge_\sigma^k V}\otimes \mathfrak{m})(\mathrm{id}_{\bigwedge_\sigma^k V}\otimes \mathfrak{m}\otimes \mathrm{id}_{\mathfrak{A}})(F\otimes \mathrm{id}_{\mathfrak{A}}^{\otimes 2})(G\otimes \mathrm{id}_{\mathfrak{A}}) H\\[3pt]
&=(\mathrm{id}_{\bigwedge_\sigma^k V}\otimes \mathfrak{m})(\mathrm{id}_{\bigwedge_\sigma^k V}\otimes\mathrm{id}_{\mathfrak{A}}\otimes \mathfrak{m})(F\otimes \mathrm{id}_{\mathfrak{A}}^{\otimes 2})(G\otimes \mathrm{id}_{\mathfrak{A}}) H\\[3pt]
&=(\mathrm{id}_{\bigwedge_\sigma^k V}\otimes \mathfrak{m})(F\otimes \mathrm{id}_{\mathfrak{A}})(\mathrm{id}_{\bigwedge_\sigma^k V}\otimes \mathfrak{m})(G\otimes \mathrm{id}_{\mathfrak{A}}) H\\[3pt]
&=F\circ(G\circ H).
\end{align*}On the other hand, it follows from the definition immediately that $I_k\circ F=F=F\circ I_k$.\end{proof}

We can extend the composition product $\circ$ to $\mathcal{E}(\mathfrak{A})$ in a natural way. For any $\mathbf{F}=(F_0, F_1,\ldots),\mathbf{G}=(G_0,G_1,\ldots)\in \mathcal{E}(\mathfrak{A})$ with $F_k,G_k\in \mathcal{E}^k(\mathfrak{A})$, we define $$\mathbf{F}\circ \mathbf{G} =(F_0\circ G_0,F_1\circ G_1,\ldots).$$Thus $(\mathcal{E}(\mathfrak{A}),\circ)$ is a unital associative algebra whose unit is $\mathbf{I}=(I_0,I_1,\ldots)$.

\begin{lemma}Let $\sigma$ be a braiding of Hecke type. Suppose that $F,G\in \mathcal{E}^1(\mathfrak{A})^\sigma$ and $(F\circ G)^{\times k}=F^{\times k}\circ G^{\times k}$ for any $k\geq 1$, then $$(F\circ G)^{\ast k}=F^{\ast k}\circ G^{\ast k}.$$\end{lemma}
\begin{proof}Observe that \begin{align*}
(F\circ G)^{\times 2}\sigma&=(F^{\times 2}\circ G^{\times 2})\sigma\\
&=(\mathrm{id}_{\bigwedge_\sigma^2 V}\otimes \mathfrak{m})(F^{\times 2}\otimes \mathrm{id}_{\mathfrak{A}})G^{\times 2}\sigma\\
&=(\mathrm{id}_{\bigwedge_\sigma^2 V}\otimes \mathfrak{m})(F^{\times 2}\otimes \mathrm{id}_{\mathfrak{A}})(\sigma\otimes \mathrm{id}_{\mathfrak{A}}) G^{\times 2}\\
&=(\mathrm{id}_{\bigwedge_\sigma^2 V}\otimes \mathfrak{m})(\sigma\otimes \mathrm{id}_{\mathfrak{A}}\otimes \mathrm{id}_{\mathfrak{A}})(F^{\times 2}\otimes \mathrm{id}_{\mathfrak{A}}) G^{\times 2}\\
&=(\sigma\otimes \mathrm{id}_\mathfrak{A})(F\circ G)^{\times 2}.
\end{align*}Thus by Proposition \ref{k-fold of quantum shuffle product}, $$(F\circ G)^{\ast k}=(F\circ G)^{\times k}=F^{\times k}\circ G^{\times k}=F^{\ast k}\circ G^{\ast k},$$as desired.\end{proof}

\begin{theorem}Keep the assumptions in the above lemma. If $\sigma$ is of finite rank, then we have $$\det\nolimits_q(F\circ G)=(\det\nolimits_q F)(\det\nolimits_q G).$$ \end{theorem}
\begin{proof}It follows immediately from the above lemma.\end{proof}

\begin{example}For any $F,G\in \mathcal{E}^1(\mathfrak{A})^\sigma$, we write $F(v_j)=\sum_{i=1}^Nv_i\otimes a^i_j$ and $G(v_j)=\sum_{i=1}^Nv_i\otimes b^i_j$ for some $a^i_j,b^i_j\in \mathfrak{A}$. If all $a^i_j$'s commute with all $b^k_l$'s, then \begin{align*}
\lefteqn{(F\circ G)^{\times m}(v_{i_1}\otimes \cdots\otimes v_{i_m})}\\
&=\sum_{k_1,j_1=1}^N\cdots \sum_{k_m,j_m=1}^Nv_{k_1}\otimes \cdots\otimes v_{k_m}\otimes a^{k_1}_{j_1}b^{j_1}_{i_1}\cdots a^{k_m}_{j_m}b^{j_m}_{i_m}\\[3pt]
&=\sum_{\substack{ 1\leq j_1,\ldots,j_m,\leq N ,\\1\leq k_1,\ldots,k_m,\leq N  }}v_{k_1}\otimes \cdots\otimes v_{k_m}\otimes (a^{k_1}_{j_1}\cdots a^{k_m}_{j_m})(b^{j_1}_{i_1}\cdots b^{j_m}_{i_m})\\[3pt]
&=(F^{\times m}\circ G^{\times m})(v_{i_1}\otimes \cdots\otimes v_{i_m}).
\end{align*}Hence we have $\det\nolimits_q(F\circ G)=(\det\nolimits_q F)(\det\nolimits_q G)$. This extends a similar result for $q$-matrices (see, e.g., Corollary 1.4 in \cite{Ta}).\end{example}

\section*{Acknowledgements}I would like to thank Gast\'{o}n Andr\'{e}s Garc\'{\i}a for pointing out to me the reference \cite{FG} in an earlier version of the manuscript. I am grateful to the referees for their useful comments, especially for pointing out to me the works \cite{GS, IO, IOP, Wo}.

\end{document}